\documentclass[twoside,12,english]{amsart}
\usepackage{color}
\usepackage[utf8]{inputenc}
\usepackage[T1]{fontenc}
\usepackage{lmodern}
\usepackage{geometry} % see geometry.pdf on how to lay out the page. There's lots.
\usepackage{babel}
\usepackage{easybmat}
\usepackage{amsmath}
\usepackage{amscd}
\usepackage{rotating}
\usepackage{tikz-cd}
\usepackage{pgffor}
\usepackage{blkarray}
\usepackage{array}
\usetikzlibrary{intersections}
\usepackage{amscd}
\usepackage{tikz}
\usetikzlibrary{matrix,arrows}
\usepackage{amsbsy}
\usepackage{amssymb,latexsym}
\usepackage{amsmath,amsthm}
\usepackage{mathtools}
\usepackage{graphicx}
\usepackage{epstopdf}
\usepackage{multicol}
\newtheorem{defn}{Definition}

\newtheorem{thm}{Theorem}
\newtheorem{prop}{Proposition}
\newtheorem{lem}{Lemma}
\newtheorem{cor}{Corollary}
\newtheorem*{thm*}{Theorem}
\theoremstyle{remark}
\newtheorem{rem}{\textsc{Remark}}

\newtheorem{example}{\textsc{Example}}

 \def\Rr{{\mathbb{R}}}

 \newcommand{\g}{\mathfrak{g}} 
  \newcommand{\m}{\mathfrak{m}} 
 \newcommand{\Mm}{\mathcal{M}}
 \newcommand{\Pa}{\mathfrak{P}}
 \newcommand{\Bp}{B\mathfrak{P}}
  \newcommand{\Rp}{R\mathfrak{P}}
 
 \newcommand{\GS}{\mathcal{GS}}
  \newcommand{\GG}{\mathcal{GG}}
 \newcommand{\F}{\mathcal{F}}
 
\def\C{\mathbb{C}}
\def\Pc{\mathbb{CP}}

\def\M{\overline{\mathcal{M}}}

\begin{document}
\title{\v Cech cover of the complement of the discriminant variety.\\ Part II: Deformations  of Gauss-skizze}
\author{N.C. Combe}
\address{Max Planck Institute for Mathematics in the Science\\ Inselstrasse 22 \\ 04 103 Leipzig}
\email{combe@mis-mpg.de}
%%%%%%%%%%%%%%
\thanks{I would like to express my gratitude towards the Max Planck society, for supporting my research through the Minerva grant. I would like to thank the Max Planck Institute for Mathematics in Bonn and in Leipzig for support, during my stay in both institutes. I wish to thank Yuri Manin for beautiful and helpful remarks and as well for his enthusiasm and interest in my work. Many thanks to Bruno Vallette for many helpful comments and advice. I would like to thank Joao Nuno Mestre for informing me about an interesting article concerning deformation theory. Finally, many thanks to the anonymous referee for suggestions and comments, which stimulated progress.}
\maketitle
\begin{abstract}
The configuration space of $n$ marked points on the complex plane is considered. We investigate a decomposition of this space by so-called Gauss-skizze i.e. a class of graphs being forests, introduced by Gauss.
It is proved that this decomposition is a semi-algebraic topological stratification. It also forms a cell decomposition of the configuration space of $n$ marked points.  
 Moreover, we prove that classical tools from deformation theory, ruled by a Maurer--Cartan equation, can be used only locally for Gauss-skizze. We prove that the deformation of the Gauss-skizze is governed by a Hamilton--Jacobi differential equation. This gives developments concerning Saito's Frobenius manifold. 
Finally, a Gauss-skizze operad is introduced. It is an enriched Fulton--MacPherson operad, topologically equivalent to the little 2-disc operad. 
\end{abstract}
%\vfill\eject

\setcounter{tocdepth}{1}
\tableofcontents

\

% Boite categorie de graphes
%%boite GAusse skizze 

\section{Introduction}
The consideration of so-called configuration spaces of points on the complex (or real) line, emerged from the interest of physicists in Conformal Field theory. 
A configuration space $ \textup{Conf}_n(\C)$ of $n$ marked points on the complex plane is:
\[\textup{Conf}_n(\C) =\{(x_1,...,x_n)\in \C^{n}| x_i\neq x_j\}.\]

\smallskip 

This complex space can be identified with the space of degree $n$ complex monic polynomials in one variable and with distinct roots. In this paper, we decompose it using so-called {\it Gauss-skizze}, i.e. a set of curves, properly embedded in the complex plane which are the inverse images of the real and imaginary axis under a complex polynomial in degree $n$, in one variable. 

\smallskip 

The Gauss-skizze object was first introduced by C. F. Gauss in~\cite{Ga73}, as a tool to prove the fundamental theorem of algebra. Recently in 2007, it was rediscovered by J. L. Martin, D. Savitt, and T. Singer in~ \cite{SaSi07}, a few years later (in 2011) by F. Bergeron ~\cite{Be11}. More recently, in 2016, E. Ghys in his "singular mathematical promenade" reinvestigated this tool~\cite{{Gh16},{Gh17}}. Finally, in~\cite{Ac17} N. A'Campo suggested to use the Gauss-skizze in a different flavour. In~\cite{Be11} and later in~\cite{Ac17}, it was conjectured in both papers that the Gauss-skizze decomposition forms a cell decomposition of the space of complex degree $n$ polynomials, with distinct roots such as depicted above. 

\smallskip 

This paper, presents the Gauss-skizze in a categorical language, which is completely different from the approach of~\cite{Be11} and~\cite{Ac17}. 
There exist important similarities between this {\it Gauss-Skizze} and the Grothendieck {\it dessins d'enfant}~\cite{Gro84}. A detailed study concerning the {\it Gauss-Skizze} can be found in the author's former works~\cite{{Co18a},{Co18b}} and have been a source of inspiration for the creation of {\it Dessins de vieillard} in~\cite{CMM20}.

\smallskip 

In this article, first of all, we prove that the decomposition in Gauss-skizze of the compactified space $\overline{\textup{Conf}_n(\C)}$ forms a topological stratification. Moreover, we prove that this topological stratification is a cell decomposition. We use the compactification introduced by Fulton--MacPherson~\cite{FuMa94}. This step prepares the ground for deeper investigations, concerning perverse sheafs in the flavour of Goresky--MacPherson's approach in~\cite{{GoMa81a},{GoMa81b}} and \cite{GoMa83}. 

\smallskip

Secondly, we study algebraic-geometric properties of this stratification, namely small deformations of the Gauss-skizze, in the sense of Grothendieck--Mumford--Schlessinger~ \cite{{Gro95},{Sch68}}, using the approach developed in \cite{Ia10} and~\cite{IaMa10}. Classical deformation theory (for Algebraic geometry) is based on the works of Kodaira--Spencer~\cite{KoS58}, and Kuranishi \cite{Ku68} on small deformations of complex manifolds. It was formalised by Grothendieck~\cite{Gro95} in a functorial language, along with Schlessinger~\cite{Sch68} and Artin~\cite{Ar76}.

\smallskip 

 The idea is that, with infinitesimal deformations of a geometric object, we can associate a deformation functor of Artin rings $F: {\bf Art} \to {\bf Set}$. The modern approach to the study of deformation functors, associated with geometric objects, consists in using a differential graded Lie algebra (DGLA) or, in general, using $L_{\infty}$-algebras. Once we have a differential graded Lie algebra ${\bf L}$, we can define the associated deformation functor $Def_{\bf L} : {\bf Art} \to {\bf Set}$, using the solutions of the Maurer--Cartan equation, up to gauge equivalence.

\smallskip

In this article, it is shown that in the Gauss-skizze context, Maurer--Cartan equations rule only {\it local} deformations. We remedy to this situation, by using a totally different approach, known as the level set method~\cite{OsPa03}. Applying the level set method to Gauss-skizze produces a non local deformation.

\smallskip

In particular, we prove that a deformation of the Gauss-skizze is governed by a (linear) differential equation of Hamilton--Jacobi type. 
Roughly speaking, the original question behind the level set method is as follows. Given an $n$-dimensional euclidean space, and an interface of codimension 1 lying within this space, and bounding a multiply connected open region $\Omega$, the problem is to analyze and compute its subsequent motion under some vector field ${\bf v}$, depending on many parameters.

\smallskip

Thirdly, we introduce a topological Gauss-skizze operad, named {\bf GaSOp}. The {\bf GaSOp} operad is homotopically equivalent to the little disc operad~\cite{BoVo68}. The advantage of {\bf GaSOp}, is that it produces a refined version of the Fulton--MacPherson operad~\cite{FuMa94}. As a corollary, we have that it defines a refined semi-algebraic stratification of the Fulton--MacPherson decomposition. Considering this configuration space as a space of polynomials, our decomposition gives also information concerning critical values and critical points of polynomials of high degree.
\smallskip

 In fact, this proposal contributes to giving a considerable amount of details about the Frobenius manifold, introduced by K. Saito in the context of unfolding isolated singularities, \cite{{Ma99},{Sa82}}.

 \smallskip 
 
 The Gauss-skizze gives a cell decomposition of this Frobenius manifold. We show that a differential equation of Hamilton--Jacobi type governs extremal paths on this Frobenius manifold.

\smallskip

The stepping stone towards the construction of this new {\bf GaSOp} is to use the categorical language introduced in~\cite{{Go},{GoW},{LaVo12},{Sin1},{Sin2}} for configuration spaces of marked points in the Euclidean space.

\smallskip

 This approach establishes a more explicit bridge between the geometric results found for the real locus of the moduli space $\M_{0,n}(\Rr)$~\cite{{Dev04},{DaJaSco},{EHKR},{Ka1}} and in the complex version $\M_{0,n}(\C)$ (see for example~\cite{{Kee83},{Kee92},{DM69},{FuMa94}}).

\smallskip

The plan of this paper is as follows. Section 1, we introduce graphs from a categorical point of view. Namely, we introduce conceptual graphs, and use the Borisov--Manin formalism to define ghost morphisms. We recall facts concerning cosimplicial and semicosimplicial models and the relation to the Fulton--MacPherson operad. 
\smallskip 

In sections 2 and 3, we introduce the Gauss-skizze stratification. Namely, the Gauss-skizze and their properties. We show that the set of (conceptual) $n$-Gauss-graphs forms a poset.

\smallskip
Section 4, contains the proof that the Gauss-skizze decomposition form a semi-algebraic stratification. We  present this decomposition in relation to the Frobenius manifolds, and prove that we have a cell decomposition.  
\smallskip

Section 5, we study the algebraic-geometric properties of the Gauss-skizze stratification and their deformations. We study the local deformations using classical tools from deformation theory. 

For a non-local deformation we proceed as follows. Consider an open region $\Omega \in \Rr^2$. We prove that the deformations of Gauss-skizze are governed by a differential equation of type: 
\[\frac{\partial\Phi({\bf x},t)}{\partial t}+{\bf v}\nabla\Phi({\bf x},t)=0,\]
where {\bf v} is a vector field, and the level set function $\Phi({\bf x},t)$ is a real function in $\Rr^2\times \Rr^+$ having the following properties:
\begin{itemize}
	\item $\Phi({\bf x},t)<0\quad$ for $x\in \Omega$,
	\item $\Phi({\bf x},t)>0\quad$ for $x\in \overline{\Omega}^c$,
	\item $\Phi({\bf x},t)=0\quad$ for $x\in \partial\Omega$.
\end{itemize}

The method used here is borrowed from the level set method. 

\smallskip 
Section 6, we study the properties of the Gauss-skizze in a categorical framework and prove that it is a topological stratification. Finally, we construct the Gauss-skizze operad, using the cosimplicial setting defined in section 2. We prove that it gives a refined semi-algebraic decomposition, compared to the classical Fulton--MacPherson decomposition.

\medskip

\section{Conceptual graphs, (semi)cosimplical spaces and Fulton--MacPherson operad}
Let {\bf Set} be the category of sets (in a finite universe). Objects of a category ${\bf C}$ are denoted Ob(${\bf C}$). 
\subsection{Category of conceptual graphs}
We work in the category of conceptual graphs {\bf CoGr}. The objects in this category and their morphisms are defined below. 
\begin{defn}[Conceptual graphs]
A conceptual graph $G$ consists of the following quadruple \[G=(E(G),\, V(G);\, \partial_G:E(G)\to V(G)\times (G),\,i_G:V_G\hookrightarrow E(G)),\]
where $E(G)$ is the set of edges of $G$, $V(G)$ is the set of vertices of $G$, $V(G)\times V(G)$ is the set of unordered pairs of vertices of $G$, $\partial_G$ is the incidence map from the set of edges to the unordered pairs of vertices, $i_G$ is the inclusion map of the vertex set into the edge set. The interior of the edges is the set given by $E(G)\setminus i_G(V(G))$.
\end{defn}
Whenever it is clear which graph is considered, we write $E$ and $V$ instead of $E(G)$ resp. $V(G)$.

\smallskip

We comment on this definition, before introducing the morphisms between graphs. 
This definition will be slightly modified in the flavour of the terminology of~\cite{Ma99}, whenever needed. The missing ingredient in the above definition is the notion of flags $Fl$. This finite set is deeply related to the set of vertices $V$ throughout the map $\delta:Fl\to V$ and the involution $j:Fl\to Fl$ where $j^2=id$ defines edges: two element orbits of $j$ form the set of edges, whereas one-element orbits form the set of tails. 

\smallskip 

For the reader who is unfamiliar with this terminology, it is convenient to think of those graphs in terms of their geometric realisations. For each vertex $v\in V$, put $Fl=\delta^{-1}(v)$ and consider the topological space star of $v$ consisting of semi-intervals having one common boundary point. Take the union of all stars and replace every two-element orbit of $j$ by a segment joining the respective vertices so that these two flags become half-edges of the edge. 
A graph is connected whenever its geometric realisation is also. 

\smallskip

The morphisms between those graphs are as follows. 

\medskip

\begin{defn}
A graph (homo)morphism $f:G\to H$ of conceptual graphs $G$ to $H$ is given by a morphism $f_E:E(G)\to E(H)$ such that the restriction $f_V=f_{P|_{V(G)}}$ to $V$ preserves incidence relations. In particular, $\partial_H(f_E(e))$ is formed by the pair of vertices $(f_V(x),f_V(y))$, whenever $\partial_G(e)$ is given by $(x,y)\in V(G)\times V(G) $, for $e\in E(G)$. In other words, we have $f_E\,\circ i_G=i_H\, \circ f_V$. 
\end{defn}

\medskip

Note that the morphisms here are not strict i.e. edges of a graph $G$ are not necessarily mapped to an edge of the graph $H$. In particular, this allows to naturally map $G$ to a graph obtained by contraction of edges. 

\medskip

\begin{defn}
A finite sequence of vertices $v_0,...,v_n$, $n\in \mathbb{N}^{*}$ such that there exists $e_i\in E(G)$ verifying $\partial_G(e_i)=(v_{i-1},v_{i})$ is called a path in $G$ .\end{defn}

A variant of this category of conceptual graphs {\bf CoGr} is the category of {\it decorated} conceptual graphs {\bf DCoGr}.

\medskip

\begin{defn}
We call a graph $G$ decorated if there exists a map from the set $E(G)$ to a finite set (of colors and/or orientations).
\end{defn}

\medskip

\subsection*{Category of trees}
We restrict our attention to the subcategory {\bf DTrCoGr} of {\bf DCoGr}, where the objects (graphs) verify the following properties:

\smallskip 

\begin{enumerate}
\item There exists at most one edge for a pair of vertices.
\item Graphs are loopless i.e.: for any vertex $v\in V(G)$ there exists no edge $e$ such that $\partial_G(e)$ is given by the pair $(v,v)\in V(G)\times V(G)$. In other words, we restrict our attention to those graphs having an injective incidence map $i_G$.
\item Graphs are acyclic i.e. for any $n\in \mathbb{N}^{*}$ and path on a finite set of vertices $v_0,...,v_n$, there exists no edge $e_n\in E(G)$ satisfying $\partial_G(e_n)=(v_{n},v_{0})$. 
\item The graph is decorated. 
\end{enumerate}
\medskip

\begin{defn}
Graphs verifying the above property are called trees.
\end{defn}

\medskip

We state the well known result that:
\begin{lem} 
The category of trees {\bf TrCoGr} endowed with the disjoint union operation $\sqcup$ form a symmetric monoidal category. 
\end{lem}
A disjoint union of trees will be called a forest. 

\medskip
\subsection*{Ghost graphs and ghost morphisms}
Here we will use the terminology introduced by Borisov--Manin~\cite{BoMa08}, and also used in the context of Feynman categories in~\cite{KaWa13}.

\smallskip
We enrich the previous graph morphism definition for $({\bf TrCoGr},\sqcup)$ as follows. 
Let $G$ be an object in $({\bf TrCoGr}, \sqcup)$. Consider in $G$ two adjacent trees $T_1$ and $T_2$.
We choose a pair of edges (or a pair of nodes which are of valency greater than 1) $e_1$ and $e_2$ (resp. $v_1$ and $v_2$), respectively in $T_1$ and $T_2$ in such a way that an edge, connecting the interiors of the edges $e_1$ and $e_2$ (resp. the vertexes $v_1$ and $v_2$), can be artificially added. 

\medskip

\begin{defn}\label{D:ghost}
Inserting an artificial edge to a given graph, such as described above, is called a ghost morphism and results in a ghost graph. 
\end{defn}
\medskip

Applying the contracting morphism operation to the ghost edge (morphism defined earlier in definition\, \ref{D:ghost}), results in retracting the ghost edge to a point. 
This operation merges both trees $T_1$ and $T_2$ into one new tree $H$. 

\medskip

\begin{prop}
The composition of a ghost operation with a classical morphism of graphs is a well defined morphism in $({\bf TrCoGr},\sqcup)$.
 \end{prop}
 \begin{proof}
 We rely mainly on the exposition of Kaufmann--Ward on Feynman categories in~\cite{KaWa13} and the Borisov--Manin formalism~\cite{BoMa08}.
 \end{proof}
 
\smallskip 

\begin{rem}
These new morphisms are also valid for $({\bf DTrCoGr},\sqcup)$, with the condition that the ghost morphism is applied only for edges decorated by the same color. Respectively, we ask that the vertices $v_1$ and $v_2$ are incident to a set of edges which are both decorated by the same set of colors. 
\end{rem}

\medskip

\subsection*{Category of rooted trees}
In the category of trees, we distinguish a class of trees called rooted trees. A rooted tree $T$ is a non-empty, connected oriented graph without loops (oriented or not) and with a {\it root}. 

\smallskip 

Edges of rooted trees are classified into two classes. 

1) The {\it external edges}: those edges, bounded by one inner vertex.

2) The {\it Internal edges}: all the other edges (i.e. those bounded by inner vertices at both ends) 

\smallskip

Any rooted tree has a {\bf unique outgoing} external edge, called the {\bf output or the root} of the tree and 
{\bf several ingoing} external edges, called {\bf inputs} or leaves of the tree.

\smallskip 

Let ${\bf RT}$ denote the category of rooted trees.
Consider $T$ and $T'$ two trees in $Ob({\bf RT})$. By a morphism from $T$ to $T'$ we understand a continuous surjective map $f: T \to T'$ with the following properties:
\begin{enumerate}
\item $f$ takes each vertex to a vertex and each edge into an edge or a vertex. 
\item $f$ preserves the orientation (if trees are oriented).
\item The inverse image of any point of $T'$ under $f$ is a connected subtree in $T$. 
\end{enumerate}

\smallskip

Both the collection of leaves in a rooted tree and the collections of edges with a given terminal vertex are ordered, using the clockwise orientation of the plane.

\medskip

\subsection{Fulton--MacPherson operad and Little disc operad}
A classical (May) operad $\mathcal{P}$ is an algebra consisting of a collection of objects $\{\mathcal{P}(n)\}_{n\geq1}$ in ${\bf C}$, such that the symmetric group $\mathbb{S}_r$ acts on $\mathcal{P}(r)$, composition maps:
		\[ \circ_i: \mathcal{P}(k)\times \mathcal{P}(l)\to \mathcal{P}(k+l-1),\] 
		
	and a unit morphism $\eta: {\bf 1}\to \mathcal{P}(1)$, satisfying equivariance, unit, associativity axioms (\cite{Ma72}, see~\cite{Fr17} for a more modern exposition).

\smallskip

According to~\cite{BoMa08}, an operad can be seen also a functor from the symmetric monoidal category $({\bf RT},\sqcup)$ of labeled graphs (the rooted trees) to a symmetric monoidal category $(G, \otimes)$, which is called the ground category. For any tree $T$, we associate the object $\mathcal{P}(T)=\otimes_{v\in T} \mathcal{P}(In(v))$, where $In(v)$ is the set  of input edges of a vertex $v$.

\smallskip 
Given a rooted tree $T$ and a set of marked  inputs, the composition of $T'$ with $T$ is done by graphting the output of $T'$ with the $i$-th input of $T$. The unit element is the rooted tree with no leaves.
Given two trees $T$ and $T'$, we have  a functorial isomorphism between  $\mathcal{P}(T\circ_i T')$ and $\mathcal{P}(T)\otimes \mathcal{P}(T')$.
 \smallskip
 
 %%%%%%%%%%%%%%%%%%%%%%%%%%%%%

The operad that we investigate here is the Fulton--MacPherson operad~\cite{FuMa94},\cite{GeJo94}, where the collection of objects $\{\overline{\textup{Conf}}_n(\Rr^d)\}_{n\in \mathbb{N}}$ are configuration spaces of $n$ marked points on $\Rr^d$. Our attention is restricted to $\Rr^2$.

 \smallskip
 
 The Fulton--MacPherson operad and the little disc operad, introduced by Boardman and Vogt~\cite{{BoVo68},{BoVo73}}, are homotopy equivalent (see~\cite{Sa98}, Prop.4.9).
The Fulton--MacPherson compactification of the configuration space of ordered $n$-tuples of points in the Cartesian space/Euclidean space $\Rr^d$ is the topological closure of the image of:

\smallskip 

\[FM_n(\Rr^d):=\mathfrak{i}(Emb([n], \Rr^d)/(\Rr^d\rtimes \Rr_{>0})\]

where $\mathfrak{i}$ is a continuous function from the ordered configuration space into the Cartesian product of $(d-1)$-spheres and intervals: $ (\mathbb{S}^1)^{n(n-1)}\times[0,\infty)^{n(n-1)(n-2)}$. This construction will be more precisely exposed in the section~\ref{S:conf}. The group of translation and scale transformation $\Rr^d\rtimes \Rr_{>0}$ acts canonically on $\Rr^d$. 

\smallskip 

 The symmetric group $\mathbb{S}_n$ acts canonically on $FM_n(\Rr^d)$ and so Fulton--MacPherson compactifications of the configuration space of points $\overline{\textup{Conf}}_n(\Rr^d)$ is given by \[\overline{\textup{Conf}}_n(\Rr^d)=FM_n(\Rr^d)/\mathbb{S}_n.\]

\smallskip

We now describe little disc operadt briefly below. 

\smallskip 

Consider a unit disc $D$ in $\C$ and let $O(n)$ be a topological space. 
\[O(n) =\left\{\begin{pmatrix} 
z_1 & \dots &z_n \\
r_1 &\dots & r_n 
\end{pmatrix} \in \begin{pmatrix} D^n \\ \Rr^n_{+}\end{pmatrix} \mid \text{the discs}\quad r_iD+z_i \quad \text{are disjoint subsets of} \quad D \right\}. \]

\smallskip

The symmetric group $\mathbb{S}_n$ acts on $\Pc^n$ by permuting the discs. 
The product operation in this operad is defined by glueing disks.

\smallskip 

If \[a= \begin{pmatrix} 
z_1 & \dots &z_k \\
r_1 &\dots & r_k
\end{pmatrix}\quad \text{and} \quad b_i= \begin{pmatrix} 
w^i_1 & \dots &w^i_{n_i} \\
s^i _1 &\dots & s^i_{n_i} 
\end{pmatrix}, \]
then, the final output of the product operation is: 
\[\begin{pmatrix} 
r_1w^1_1+ z_1 & \dots & r_1w^1_{n_{1}}+ z_1 & \dots & r_kw^k_{1}+ z_1& r_kw^k_{n_{k}}+ z_k\\
r_1s^1_1 &\dots & r_1s^1_{n_1} &\dots & r_ks^k_{1} &r_ks^k_{n_k}
\end{pmatrix}.\]
 
The map from the topological spaces $O(n)$ to the configuration space $\overline{\textup{Conf}}_k(\C)$ defined by \[ \begin{pmatrix} 
z_1 & \dots &z_k \\
r_1 &\dots & r_k
\end{pmatrix} \to (z_1, \dots , z_k)\] is a homotopy equivalence. 

\smallskip

The relation between the little disc operad and the Fulton--Macpherson operad is stated in the following proposition. 
\begin{prop}[\cite{LaVo12}, Prop.\,5.6 and \cite{Sa98}] 
The Fulton--MacPherson operad $FM_n(\Rr^d)$ is weakly equivalent in the model structure on operads with respect to the classical model structure on topological spaces, to the little $d$-disk operad.\end{prop}

%%%%%%%%%%%%%%%%%%%%%%%%%%%%%%
\subsection{Cosimplical and semicosimplical models}~\label{S:1}
One can investigate those configuration spaces using cosimplicial models. We borrow the cosimplicial model which has been constructed for the space of long knots $\mathbb{R}\hookrightarrow \mathbb{R}^d$, for $d\geq 4$ (presented in~\cite{{Sin1}, {Sin2},{LaVo12}}) and apply it to the more specific case of configurations spaces with marked points on the complex plane.

\smallskip 

We briefly recall some general facts about cosimplicial and semicosimplicial models. For more details we recommend~\cite{Bo} and~\cite{BoKa72}.

\smallskip 

Let $\Delta_{\bullet}$ be the a category of finite ordinal numbers, whose objects are finite ordinal sets $[n]=\{0,1,...,n\}$, $n\in \mathbb{N}$; morphisms are order-preserving injective maps among them. Every morphism in $\Delta_{\bullet}$, different from the identity, is a finite composition of coface morphisms:
\begin{equation}
\partial_k:[i-1]\to [i], \quad \partial_k(p)=\begin{cases}
    
    p\quad if \quad p<k \\
    p+1\quad if \quad k\leq p
   \end{cases}
  for \quad k=0,\dots, i
   \end{equation}
Relations about compositions are $\partial_l\partial_k= \partial_{k+1}\partial_l$ , for every $l \leq k$.

\smallskip 

Given a category ${\bf C}$, a cosimplicial object $C^\bullet$ in ${\bf C}$ is a functor from $\Delta_{\bullet}$ to {\bf C}. If $S$ is a nonempty finite totally ordered set then $C^S$ will denote the value of $C^\bullet$ at $S$, except that we write $C^m$ instead of $C^[m]$. 
The collection of objects $C^n\in {\bf C}$ for $n\geq 0$, are endowed with {\it coface} maps $d^i: C^{n}\to C^{n+1}$ for $1\leq i \leq n$ and {\it codegeneracy} maps $s_j: C^{n+1}\to C^n$, $1\leq j \leq n$, satisfying the usual cosimplicial identities (~\cite{BoKa72} p.\,267). To a cosimplicial object over ${\bf C}^{op}$, corresponds a cosimplical object over ${\bf C}$. 

\smallskip

By simplicial category, is meant a category ${\bf C}$ enriched over a category ${\bf S}$, and we write $hom(X,Y) \in S$ for the mapping space of $ X,Y\in {\bf C}$. By a simplicial model category, we mean a model category ${\bf C}$ which is also a simplicial category satisfying axioms:
\smallskip 

\begin{enumerate}
\item the object $X\otimes K$ and $hom(K,X)$ exist for each $X\in {\bf C}$ and each finite $K \in {\bf S}.$
\item If $i: A\to B\in {\bf C}$ is a cofibration and $p: X\to Y \in {\bf C} $is a fibration , then the map
\[(i,p): hom(B, X)\to hom(A,X) \times_{hom(A,Y)}hom(B,Y),\]
is a fibration, which is trivial if either $i$ or $p$ are trivial.
\end{enumerate} 
As is mentioned in~\cite{Bo}, when ${\bf C}$ is a model category, there is an induced model category structure on the category of cosimplicial objects over ${\bf C}$.

\smallskip 

A semicosimplicial object in a category ${\bf C}$ is a covariant functor $B^{\Delta}:\Delta_{\bullet} \to {\bf C}$(see~\cite{{EZ50}, {We94}}). In other words, we can say that a semicosimplicial object $B^{\Delta}$ is a
diagram in ${\bf C}$ such that:

\[ B_0 \mathrel{\substack{\textstyle\longrightarrow\\[-0.2ex]
      \textstyle\longrightarrow \\[-0.2ex]}}
      B_1 \mathrel{\substack{\textstyle\longrightarrow\\[-0.2ex]
     \textstyle\longrightarrow \\[-0.2ex]
      \textstyle\longrightarrow}}
      B_2 \mathrel{\substack{\textstyle\longrightarrow\\[-0.2ex]
      \textstyle\longrightarrow \\[-0.2ex]
       \textstyle\longrightarrow \\[-0.2ex]
      \textstyle\longrightarrow}} \dots\]

where each $B_i$ is in {\bf C}, and, for each $i>0$, there exist $i+1$ morphisms
$\partial_k:B_{i-1} \to B_i,$ $k=0,...,i$, such that $\partial_l\partial_k = \partial_{k+1}\partial_l$, for any $l \leq k.$

\smallskip 

This definition is used as follows. Let $X$ be a smooth variety, defined over an algebraically closed field of characteristic 0. Consider an affine open cover $U = \{U_i\}$ of $X$, and $\Theta$ a sheaf of Lie algebras on $X$. 

\smallskip 
Then, we can define the \v Cech semicosimplicial Lie algebra $\Theta(U)$ as the semicosimplicial Lie algebra:

\[ \Theta(U):\quad \prod_i\Theta(U_i) \mathrel{\substack{\textstyle\longrightarrow\\[-0.2ex]
      \textstyle\longrightarrow \\[-0.2ex]}}
      \prod_{i<j}\Theta(U_{ij}) \mathrel{\substack{\textstyle\longrightarrow\\[-0.2ex]
     \textstyle\longrightarrow \\[-0.2ex]
      \textstyle\longrightarrow}}
     \prod_{i<j<k}\Theta(U_{ijk}) \mathrel{\substack{\textstyle\longrightarrow\\[-0.2ex]
      \textstyle\longrightarrow \\[-0.2ex]
       \textstyle\longrightarrow \\[-0.2ex]
      \textstyle\longrightarrow}} \dots\]

where the coface maps are defined as follows: 
\[\partial_k: \prod_{i_0<\dots<i_{h-1}}\Theta(U_{i_0\dots i_{h-1}})\to \prod_{i_0<\dots<i_{h}}\Theta(U_{i_0\dots i_{h}}) \]
and are given by \[\partial_k(x)=x_{i_0\dots\hat{i}_k\dots i_{h|_{U_{i_0\dots i_{h}}}}},\quad \text{for}\quad k=0,...,h.\]

\medskip 

\subsection{Cosimplical model for configuration spaces and operads}\label{S:conf}
The definition of operads is used in the cosimplical setting. Recall the $\circ_i$ operations $\circ_i: \mathcal{P} (n)\circ \mathcal{P}(m) \to \mathcal{P}(n+m-1)$, which provide a basic set of morphisms for an operad. following section 3 of \cite{MS02}, we have the following adapted definition of operads.

\medskip

\begin{defn}[section 3 of \cite{MS02}]\label{D:co} 
Given an operad $\mathcal{P}$ with multiplication $\circ$, let $\mu$ denote the morphism $1_{\bf C} \to \mathcal{P}(2)$.
We define the coface maps $d^i: \mathcal{P}(n)\to \mathcal{P}(n+1)$ by

\begin{equation}
 d^i =
 \begin{cases}
  1_{\bf C} \circ \mathcal{P}(n) \xrightarrow{\mu\circ id} \mathcal{P}(n)(2)\circ \mathcal{P}(n) \xrightarrow{\circ_0} \mathcal{P}(n+1) &\text{if}\quad i=0\\
  \mathcal{P}(n) \circ 1_{\bf C} \xrightarrow{id\circ \mu} \mathcal{P}(n)\circ \mathcal{P}(2) \mathcal{P}(n+1) \xrightarrow{\circ_i} \mathcal{P}(n+1) & \text{if}\quad 0<i<n+1 \\
  1_{\bf C}\circ \mathcal{P}(n) \xrightarrow{\mu\circ id} \mathcal{P}(2)\circ \mathcal{P}(n) \mathcal{P}(n+1) \xrightarrow{\circ_{n+1}} \mathcal{P}(n+1) & \text{if}\quad i=n+1 
  \end{cases}  
  \end{equation}
  
  \smallskip
  
  The codegeneracy map $s_i$ is defined as $\mathcal{P}(r_i)$, where $r_i:T_n\to T_{n-1}$ contracts the $i$-th leaf of the tree $T_n$. 
  Let $\mathcal{P}^{\bullet}$ be the cosimplicial object in ${\bf C}$ whose $n$-th entry is $\mathcal{P}(n)$ and whose coface and codegeneracy
maps are given by $d_i$ and $s_i$ above.
\end{defn} The maps $d_i$ and $s_i$ satisfy cosimplicial identities.

\medskip

We discuss the details of the construction of the little disc operad in the flavour of~\cite{LaVo12} and~\cite{Sin2} and sketch the construction. 

\smallskip 

Consider $M$ to be the 2-dimensional disc (with non-empty boundary). The $n$-th entry of the considered cosimplicial model is given by the Cartesian product $M^n$. Now, fix two unit tangent vectors $\alpha \in UTM$ (resp. $\beta\in UTM$), located on the boundary of $M$, pointing inward (resp. outward). 

\smallskip 

We consider the configurations spaces of marked points on $M$, denoted $\textup{Conf}_n(M)$. Elements of $UTM$---the unit tangent bundle to $M$---are denoted by $\xi =(x,\overrightarrow{v})$, with $x \in M$ and $\overrightarrow{v}\in UT_xM$ with $||\overrightarrow{v}||=1$.

\smallskip 

We consider the configuration space $\textup{Conf}_{n+2}(M)$, consisting of $(n+2)$-tuples, \[\xi_0=(x_0,\overrightarrow{v}_0),\xi_1=(x_1,\overrightarrow{v}_1), \dots ,\xi_{n+1}=(x_{n+1},\overrightarrow{v}_{n+1}))\in (UTM)^{n+2}\] such that $\xi_0=\alpha, \xi_{n+1}=\beta$ and $x_i \neq x_j$. We will work in the framework of the Fulton--MacPherson compactification $\overline{\textup{Conf}}_{n+2}(M)$. 

\smallskip

The idea of the construction is that elements of $\overline{\textup{Conf}}(M)$ consist of some ``virtual'' configurations, where the marked points $x_i$ and $x_j$ may be equal---in which case, some extra data serves to distinguish two points infinitesimally close---as explained in the definitions 4.1 and 4.12 of~\cite{Sin1}).

\smallskip

The following crucial step concerns the construction of the doubling maps, for this cosimplicial model. 
Let $0\leq i \leq q$:
 \begin{equation}
 d^i:\overline{\textup{Conf}}_{q-1}(M)\to \overline{\textup{Conf}}_{q}(M)
 \end{equation}
\[(\xi_0,..., \xi_i,...,\xi_q)\mapsto (\xi_0,..,\xi_i, \xi'_i,....,\xi_q),\]
where $\xi'=(x_i',\overrightarrow{v}_i)$ with $x_i=x'_i$, but infinitesimally $x_i'-x_i=\overrightarrow{v}_i$. 

\smallskip 

The forgetting maps are defined as follows:
\begin{equation}s_j: \overline{\textup{Conf}}_{q}(M)\to \overline{\textup{Conf}}_{q-1}(M),
 \end{equation}
 
 \[ (\xi_0, ..., \xi_i,..., \xi_q)\mapsto (\xi_0, ..., \hat{\xi_i},..., \xi_q),\] 
where $\hat{\xi_i}$ means that this point is ``forgotten''.

\begin{rem}
As was pointed out in~\cite{Sin1} (definition 6), by defining forgetting maps $s_j$ and $d_i$ as cofaces, certain cosimplicial identities are not satisfied. To remedy, it is possible to replace $\overline{\textup{Conf}}_{q}(M)$ by a homotopy equivalent quotient, for which the induced map satisfy the cosimplicial identities. This latter space is denoted by $C'\langle[M, \partial]\rangle$.
The cosimplicial space is given by \[X_{\bullet}\{C'\langle[M, \partial]\rangle, d_i, s_j\}_{q\geq 0}.\]
\end{rem}

\smallskip
 To have the cosimplicial model, two important tools are introduced: the so-called maps $\theta$ and $\delta$. Let $S\in {\bf Fin}$ be a finite set of points. Let $\tilde{x}$ be a given configuration.  

\begin{enumerate} 

\item Take two (distinct) elements $a$ and $b$ in $S$: 

\[\theta_{a,b}: \textup{Conf}_S(M)\to \mathbb{S}^{1}\]

\[\theta_{a,b}: \tilde{x}\mapsto \frac{\tilde{x}(b)-\tilde{x}(a)}{||\tilde{x}(b)-\tilde{x}(a)||}.\]
This map gives the direction between vectors given two points of the configuration $x$. 

\item For three distinct elements $a,b,c$ in $S$, we define their relative distance map: 

\[\delta_{a,b,c}:\textup{Conf}_S(M)\to [0,+\infty)\]
\[ \tilde{x}\mapsto \frac{||\tilde{x}(a)-\tilde{x}(b)||}{||\tilde{x}(a)-\tilde{x}(c)||}.\]

\smallskip

For three different points $a,b,c$ in $S$ and $x\in \overline{\textup{Conf}}_S$ we adopt
the notation: \[\tilde{x}(a)\cong \tilde{x}(b) \ \mathrm{ rel}  \ \tilde{x}(c),\] if their relative distance is zero. 
\end{enumerate}

\smallskip

\begin{rem} Note that up to translations and dilatation, any configuration can be recovered from the direction $\theta$ and the relative distances $\delta_{a,b,c}$.
\end{rem}

\begin{prop} 
Let $\tilde{x}$ be a given point in the Fulton--MacPherson compactified configuration space. We have the following equivalent conditions:

\begin{itemize}
\item $\tilde{x}$ lies in the boundary of the configurations space, 
\item (if and only if) there exist three points $a,b,c$ in $S$, such that $a$ and $b$ are relatively close, with respect to $c$. This is denoted by $\tilde{x}(a)\cong \tilde{x}(b) \ \mathrm{ rel}  \ \tilde{x}(c)$.
\end{itemize}
\end{prop}

%%%%%%%%%%%%%%%%%%%%%%%%%%%%%%%%%%%%%%% 
%OPerad 
\subsection{Weak partitions and totally ordered sets}\label{S:op}
For the operadic composition operation, on the configuration space, it is necessary to introduce {\it weak partitions}. We follow some notions depicted in~\cite{LaVo12}. 

\smallskip 

Fix two objects: a finite set $S\in {\bf Fin}$; a weak partition $\tilde{v}:S\to Q$, where $Q$ is a linearly ordered finite set. A {\it linearly ordered} set (or a totally ordered set) is a pair $(\mathcal{L},\leq)$, where $\mathcal{L}$ is a set and $\leq$ is a reflexive, transitive, and antisymmetric relation on $\mathcal{L}$, such for any pair $x, y \in \mathcal{L}$ we have $x \leq y$ or $y \leq x$. We write $x < y$ when $x \leq y$ and $x\neq y$.

\smallskip

 A {\it weak partition of a finite set} $S$ is a map $\tilde{v}: S \to Q$ between two objects of ${\bf Fin}$, where for a given $p \in Q$, the preimages $\tilde{v}^{-1}(p)$ are elements of the partition. Note that it is not required that $\tilde{v}$ is surjective: some of the elements $\tilde{v}^{-1}(p)$ are allowed to be empty. 

\smallskip

Concerning terminology, in case $\tilde{v}$ is not surjective, the weak partition is degenerate. Otherwise, it is non-degenerate. We will simply say ``partition'' instead of a non-degenerate weak partition. The (weak) partition $\tilde{v}$ is ordered if its codomain $Q$ is equipped with a linear order. The undiscrete partition is the partition $\tilde{v} : S \to \{1\}$ whose only element is $S$. 

\smallskip

We adopt the following notations: 
\begin{equation}
Q^{*}=\{0\} \star Q, A_Q=\tilde{v}^{-1}(p), \text{and}\quad A_0 =Q. 
\end{equation}

\smallskip 

Given two disjoint linearly ordered sets $(L_1,\leq_1)$ and $(L_2,\leq_2)$ their ordered sum is the linearly ordered set $L_1 \star L_2 := (L_1 \cup L_2 , \leq)$ such that the restriction of $\leq$ to $L_i$ is the given order $\leq_i$ and such that $x_1 \leq x_2$ when $x_1 \in L_1$ and $x_2 \in L_2$. More generally, if $\{L_p\}_{p\in Q}$ is a family of linearly ordered sets, indexed by a linearly ordered set $Q$, its ordered sum $\star_{p\in Q} L_p$ is the disjoint union $\sqcup_{p\in Q}L_p$ equipped with linear order $\leq$, whose restriction to each $L_p$ is the given order on that set such that $x<y$ when $x\in L_p$ and $ y\in L_q$ with $p<q$ in $Q$. 
 
 \smallskip 
 
We have: 
\begin{equation}\prod_{p\in Q^*} \overline{\textup{Conf}}_{A_p}=\overline{\textup{Conf}}_Q\times \prod_{p\in Q} \overline{\textup{Conf}}_{\tilde{v}^{-1}(p)},\end{equation}
which defines the operad structure:
\[\Phi_{\tilde{v}}:\prod_{p\in Q^{*}} \overline{\textup{Conf}}_{A_{p}}\to \overline{\textup{Conf}}_{A}.\]

\smallskip 

As a more intuitive explanation, the configuration $x=\Phi_{\tilde{v}}((x_p)_{p\in Q^*})$ is obtained by proceeding by a replacement, the $p$-th component $x_0(p)$ of the configuration $x_0\in \overline{\textup{\textup{Conf}}}_{Q}$, for any point $p\in Q$, by the configuration $x_p\in \overline{\textup{Conf}}_{A_Q}$ made infinitesimal, see previous section and section 5.2 of~\cite{LaVo12} for more details. The case of inserting two points---being infinitesimally close---enter the description from~\cite{LaVo12}, where the cosimplicial space is defined. 

%%%%%%%%%%%%%%%%%%%%%%%%%%%%%%%

\section{Gauss-skizze stratification}
\subsection{Embedded graphs arising from complex polynomials}
\subsection*{Terminology and conventions}
\smallskip

The geometric realisation of a tree is a connected contractible 1-complex with at least one edge. We will call vertices of valency 1 leaves. A forest is a disjoint union of trees.
An embedded forest is a subset of the plane which is the image of a proper embedding of a forest minus leaves, to the plane. The notation $Card(V)$ stands for the cardinality of the set $V$ . The notation $val(v)$, where $v\in V$, indicates the valency of the vertex $v$ i.e. how many edges are incident to the vertex $v$. By {\it tails} we mean the union of a vertex and the set of half-segments incident to it.

\smallskip 
\subsection*{Level curves and Gauss-skizze}
It is known by~\cite{ErJaNa07} that if $u(x,y)$ is a harmonic polynomial in two variables, of degree $n$, then the level curve $\{(x,y)\in \mathbb{R} : u(x,y) = 0\}$ is an embedded forest. This comes from the fact that the existence of a cycle contradicts the maximum principle. All vertices of the forest, except for the leaves, have even valencies. Moreover, there are exactly $2n$ leaves.

\medskip

\begin{defn}
Color in blue (resp. red) the embedded forest given by the level set $\{(x,y)\in \mathbb{R} : ReP(x,y) = 0\}$ (resp. $\{(x,y)\in \mathbb{R} : ImP(x,y) = 0\}$).
A $B$-pattern (resp. R-pattern) is the level set $\{(x,y)\in \mathbb{R} : ReP(x,y) = 0\}$ (resp. $\{(x,y)\in \mathbb{R} : ImP(x,y) = 0\}$). The space of $B$-patterns (resp. $\Bp_n$-patterns) for degree $n$ polynomials is denoted by $\Bp_n$ (resp. $\Rp_n$). If colors are ignored, we refer to the space of patterns associated to degree $n$ polynomials as $\Pa_n$.
\end{defn} 
\medskip

\begin{defn}
\ 

\begin{itemize}
\item A Gauss-skizze is a bi-colored embedded forest, being the inverse image under the polynomial $P$ of the real and imaginary axis. The space of Gauss-skizze (for the space of degree $n$ polynomials) is denoted by $\GS_n$. 
\item A Gauss-graph is an object of the conceptual category of decorated planar forests corresponding to a given Gauss-skizze, and coding its combinatorial data. The space of $n$-Gauss-graphs is denoted by $\GG_n$.
\end{itemize}
\end{defn} 
\medskip

The properties of Gauss-skizze's are such that: 
\begin{itemize}
\item $P^{-1}(\mathbb{R})$ (resp.\,$P^{-1}(i\mathbb{R})$) gives a system of $n$ red blue (resp. red) curves, properly embedded in $\C$.
\item The curves are oriented, and each red curve intersects a unique blue curve, once. 
\item The asymptotic directions are $\frac{2k\pi}{4n}$, where $k\in \{0,...,4n-1\}$.
\end{itemize}

\begin{itemize}
\item The roots of the polynomial are at the intersections of red and blue curves. The set of roots is denoted $V_{roots}$. 
\item The critical points $z_0$ of $P$ verifying $Re(P)(z_0)=0$ (resp.\,$Im(P)(z_0)=0$) such that their associated critical values $P(z_0)$ are imaginary numbers (resp. real) are at the intersection of curves of one color. The set of those critical points is denoted$V_{crit}$. 
\end{itemize}
 \begin{rem} 
For simplicity, we have added a coloring to the complementary regions to the embedded planar graph, instead of defining an orientation of the curves. Both approaches are equivalent in the context of those planar graphs. 
\end{rem}

We color the regions of $\C\setminus \Rr\cup\imath\Rr$ in the colors $A,B,C,D$, where $A$ is the color of the rightmost quadrant in the upper-half plane and the other colors of the other regions $B,C,D$ follow respectively in counterclockwise order. We can label the regions in the complementary part of the drawing by $A,B,C,D$: each region being the inverse image of a quadrant of the complex plane. This tool gives more information about the Gauss-skizze:

\smallskip

\begin{itemize}
\item since, the roots of the polynomial are given by the intersection of a red and a blue curve, the adjacent 2-faces are respectively colored $A,B,C,D$ in the counterclockwise orientation. 
\item Critical points of real (imaginary) critical values lie at the intersection of curves of one color. The adjacent 2-faces are colored $A,D,A,D,A,D$, or $B,C,B,C,B$. For curves of the opposite color we have: either $A,B,A,B,A$, or $C,D,C,D,C,D$.
\end{itemize}

\medskip 

\subsection*{Gauss-graphs}
Formally and explicitly, Gauss-graphs verify the following definition. 
\medskip

\begin{defn}~\label{D:1}
The $n$-Gauss-graph $\sigma^n$ is a decorated planar forest in $({\bf TrCoGr},\sqcup)$, defined by a sextuple of sets and maps: 
\[(V_\sigma=V_{roots}\cup V_{crit}\cup \{*\}, E_{\sigma}, F_\sigma,\partial_{1}, \partial_{2},n),\]
where:
\begin{itemize}
\item $V_{\sigma}$ are vertices; $0 \leq Card(V_{roots})\leq n$ and $0 \leq Card(V_{crit}) \leq n-1$. 
\item $E_{\sigma}$ are (colored) edges, 
\item $F_\sigma$ are (colored) 2-cells, 
\item the map $\partial_1: F_\sigma\to V_{\sigma}$ , $\partial_2: F_\sigma\to E_{\sigma}$ are the boundary maps.
\end{itemize}
\end{defn}

We call this graph {\it decorated} since, two additional (coloring) maps exist:
 \[E_{\sigma}\to \{R,B\}\]
 \[e\mapsto e_{R}\quad (\text{resp.}\quad e_{B})\]
 
 \[F_\sigma\to \{A,B,C,D\}\]
 \[f \mapsto {A}\quad (\text{resp.} {B},{C}, {D}.)\]

Precisely, the rules of the coloring are given below: 
\begin{itemize}
\item For any $e_R\in E$, $\partial_{1}^{-1}(e_{R})$ is $\{A,D\}$ or $\{B,C\}$.
\item For any $e_B\in E$, $\partial_{1}^{-1}(e_{B})$ is $\{A,B\}$ or $\{C,D\}$. 
\item For any $v_1\in V_{roots}$, $\partial_{1}^{-1}(v_1)=\{A,B,C,D\}$, $val(v_1)=4k$, $k\in \mathbb{N}^{*}$.
\item There exist $4n$ leaves. 
\item For any $v_2 \in V_{crit}$, such that $v_2\not\in V_{roots}$, we have: $\partial_{1}^{-1}(v_2)=\{B,C\}$ or $\{A,D\}$, if edges incident to $v_2$ are colored $R$. 
\item $\partial_{1}^{-1}(v_2)=\{A,B\}$ or $\{C,D\}$, if edges incident to $v_2$ are colored $B$.
\item For any $v_2 \in V_{crit}$, such that $v_2\not\in V_{roots}$: $4 \leq val(v_2)\leq 2n$. 
\end{itemize}

We call those graphs {\it generic} if $V_{crit}$ is empty. Tails of the vertices in $V_{roots}$ are formed from $4k$ half-segments of alternating colors. Tails for vertices of $V_{crit}$ are formed from $2m$ half-segments, where $m\geq2$, and those half-edges are decorated by one fixed color.

The relation between patterns and Gauss-skizze can be summarized in the following diagram. 

\medskip

\begin{lem}
 Let $\sqcup: \Rp_n\times \Bp_n\to \Pa_n$ be a disjoint union operation on the space of patterns (i.e. a superimposition of patterns) associated to the space of degree $n$ polynomials. 
\[\begin{tikzcd}
\Bp_n \arrow[dr] &\arrow[r] \arrow[d,"h"] \Bp_n \bigsqcup \Rp_n \arrow[l] & \Rp_n\arrow[dl] \\
 &  \GS_n & 
\end{tikzcd}\]
Then, $h$ is an injection. 
\end{lem}
\begin{proof}
Indeed, not all combinations of $R$-patterns with $B$-patterns give a Gauss-skizze. The reason is that the Cauchy--Riemann equations must be satisfied and that the resulting graph must contain no cycles (it is a forest). 
\end{proof}
\smallskip 

The {\it geometric realisation} $|\sigma^n|$ of a Gauss-graph $\sigma^n$ refers to a Gauss-skizze, i.e. a pre-image of $\mathbb{C}$ under a polynomial, where the $4n$ leaves are glued at infinity. 
When the context is clear instead of referring to degree $n$ Gauss-graphs, we will simply refer to Gauss graphs. 

\medskip

\subsection*{Morphisms between Gauss-graphs}
Let us consider $\sigma$ and $\tau$ two forests in $({\bf TrCoGr},\sqcup)$, with the same number of leaves.
Degeneration morphisms (i.e. morphisms which are not strict morphisms) are done on the graphs (and therefore their geometric realisations), using so-called {\it Whitehead moves} (WH moves).
\medskip

\begin{defn}\label{D:WH}
A {\bf contracting Whitehead} move $WH_-$, in a graph in $({\bf TrCoGr},\sqcup)$ or $({\bf DTrCoGr},\sqcup)$ is a composition of a ghost morphism with a contracting morphism on a forest (or tree).
\end{defn}
There are three cases to discuss. 
\begin{enumerate}
\item Case 1: Partial contraction on a forest. 
\begin{itemize}
\item STEP [Ghost graph]: Add a $k$-gon within a 2-cell in the set $F_{\sigma}$, such that each vertex is mapped onto the interior of an edge, lying ins the boundary of $F_{\sigma}$. Edges are of the same color. 
\item STEP [Contraction morphism]: The $k$-gon is shrinked to one vertex. 
\end{itemize}
\item Case 2: Partial contraction on a tree. This is the case where we consider two vertices in $V_{crit}$. 
There is no ghost graph. The edge connecting these two vertices is shrinked onto one vertex. 
\item Case 3: Total contraction (only for Gauss-graphs). 
Here, the ghost graph is such that the vertices of this $k$-gon coincide with the inner nodes of $V_{roots}$ lying in the boundary of $F_{\sigma}$. The $k$-gon is then shrinked to one vertex. 
\end{enumerate}
\medskip

\begin{defn}
An {\bf expanding Whitehead} move $WH_{+}$, is the inverse procedure of a contracting $WH_{-}$ move. 
\end{defn}

There is a strong relationship  between the Gauss-skizze and the Gauss-graphs. 
\begin{thm}
	There is a bijection between the set of Gauss-skizze and the set of Gauss-graphs. 
\end{thm}
The proof of this theorem is due to Theorem 1.1~\cite{Ac17}.

\medskip 
%%%%%%%%%%%%%%%%%%%%%%%%%%%%%%%%%

\section{Topological stratification of $\overline{\textup{Conf}}_n(\C)$ and Saito's Frobenius manifold}
A {\it poset} (partially ordered set) is a set $\mathcal{S}$ endowed with a binary relation $\prec$
which is reflexive, transitive and anti-symmetric. 

\medskip

\begin{defn}\label{D:closure}
Consider graphs being forests. 
Let $\sigma \prec \tau$ whenever $\tau$ can be obtained from $\sigma$ by a sequence of repeated contracting Whitehead moves i.e. when $\tau$ is incident to $\sigma$. 
We call the union $\overline{\sigma}=\{\tau : \sigma \prec \tau\}$ the combinatorial closure of $\sigma$. 
\end{defn}

\medskip

\begin{lem}
The set of $n$-Gauss graphs $\GG_n$ endowed with the relation $\preceq$ forms a partially ordered set. 
\end{lem}
\begin{proof}
The relation is clearly reflexive. If $\sigma \prec \tau$, then $\tau$ is obtained from a sequence of contracting $WH_-$ moves from $\sigma$. In particular, it is impossible that $\sigma \prec \tau$ implies $\tau \prec \sigma$. So anti-reflexivity holds. As for transitivity, if $\sigma \prec \tau$ and $\tau\prec \mu$ then, $\tau$ is obtained from $\sigma$ by contracting $WH_-$ moves. Similarly, $\mu$ is obtained from $\tau$ by contracting $WH_-$ moves. So, we have $\sigma\prec\mu$.
\end{proof}

Let $\tilde{P}$ denote the full subcategory of the category of all small categories consisting of posets. Here the posets are viewed as categories in the standard way, i.e., with elements being the objects (Gauss-graphs) and order-relations ($'\prec'$) being morphisms. 

\subsection{Stratification of $\overline{\textup{Conf}}_n(\C)$}
Let ${\bf Top}$ denote the category of topological spaces and continuous maps.

We build the stratification of the space $X=\overline{\textup{Conf}}_n(\C)$. 
\medskip

\begin{defn}
Let $(S, \prec$), be a partially ordered set. A stratification of a topological space $X$ is a locally finite collection of disjoint, locally closed subsets of $X$ called strata $X_{\alpha}\subset X$ such that 
\begin{enumerate}
\item $X=\cup_{\alpha\in S}X_{\alpha}$ and 
\item $X_{\alpha}\cap \overline{X}_{\beta}\neq \emptyset \iff X_{\alpha}\subset \overline{X}_{\beta}\iff \alpha=\beta \wedge \alpha\prec \beta$
\end{enumerate}
\end{defn}

The boundary condition can be reformulated as follows: the closure of a stratum is a union of strata. A different interpretation of this would be to define a stratification as a filtration
 \[\emptyset=X^{-1}\subset X^{(0)}\subset\dots\subset X^{(n-1)}\subset X^{(n)}=X,\]
where the $X_{(\kappa)}$ are closed, and where by defining $X^{(\kappa)}:=X_{(\kappa)}\setminus X_{(\kappa-1)}$, we have $\overline{X^{(\kappa)}}=X_{(\kappa)}$.

\smallskip

A set is semi-algebraic if it can be written as a a boolean combination of polynomial equations and polynomial inequalities with real coefficients. A stratification of a set is semi-algebraic whenever it is decomposed into a union of semi-algebraic sets. 

\medskip

\begin{prop}
The Gauss-skizze decomposition of $\overline{\textup{Conf}}_n(\C)$ is a semi-algebraic stratification.
\end{prop}
\begin{proof}
Indeed, the stratification is indexed by the union of the following semi-algebraic sets: $\{(z_1,...,z_i)\in \C^i: ReP(z_i)=0,\quad P'(z_i)=0\}$ and $\{(w_1,...,w_k)\in \C^k: ImP(z_i)=0,\quad P'(w_k)\}$. 
\end{proof}
\medskip

\subsection{Frobenius manifolds}
We start with reminding some basic notation and facts from~\cite{Ma99}. Let us recall that an affine flat structure on the (super)manifold $\Mm$ is a subsheaf $T^f_{\Mm}\subset T_{\Mm}$ of linear spaces of pairwise (super)commuting vector fields, such that $T_{\Mm}=O_{\Mm} \otimes T^f_{\Mm}$. Sections of $T^f_\Mm$ are called flat vector fields. 

\smallskip

A {\it Frobenius manifold} is a quadruple $(\Mm,T^{f}_{\Mm}, g, A)$, where $\Mm$ is a supermanifold in one of the (classical) categories; $g$ is a flat Riemannian metric compatible with the structure $T^{f}_{\Mm}$; $A$ is an even symmetric tensor $A: S^{3}(T_{\Mm})\to O_{\Mm}$. 

\smallskip

This data must satisfy the following conditions	:
 
 \begin{itemize}
 \item Potentiality of $A$. There exists a function $\Phi$ verifying, for any flat vector fields $X,Y,Z$:
 \begin{equation}\label{E:A1}A(X, Y, Z) = (XYZ) \Phi.\end{equation}
 \item Associativity. The symmetric tensor $A$ together with the flat Riemannian metric $g,$ define a unique symmetric multiplication $\circ : T_M \otimes T_M \to T_M$ such that : 
 \begin{equation}\label{E:A2} A(X,Y,Z)= g(X\circ Y, Z)= g(X,Y\circ Z). \end{equation} This multiplication must be associative.
 \end{itemize}
\smallskip

In the analytic case, an affine flat structure can be described by a complete atlas whose transition functions are affine linear, one can find local coordinates $x_{a}$ such that $e_{a}=\frac{\partial}{\partial x^a}=\partial_{a}$ are commuting sets of linear independent vector fields. 
 In flat local coordinates, the equation~\ref{E:A1}, becomes $A_{abc}= \partial_{a} \partial_{b} \partial_{c}\Phi$ and equation~\ref{E:A2}, is rewritten as:
 
 \[\partial_a \circ \partial_b=\sum_{c} A_{ab}^c\partial_c,\]
where $A_{ab}^c:= \sum_{e} A_{abe}g^{ec}, (g^{ab}):=(g_{ab})^{-1}.$

\smallskip

$\Mm$ is called {\it semisimple} if an isomorphism of the sheaves of $O_\Mm$-algebras:
\[(T_\Mm, \circ)\to (O^n_\Mm, \text{componentwise  multiplication})\] exists everywhere, locally.

 The semisimplicity of the Frobenius manifold implies that, in the local basis $(e_1,...,e_n)$ of $T_\Mm$, the multiplication takes the form:
 \[(\sum f_i e_i)\circ(\sum g_j e_j)=\sum f_i g_i e_i,\]
 and in particular
 \[e_i\circ e_j = \delta_{ij}e_j.\]

In the dual basis, the metric ( $g(e_i,e_k)=\delta_{ik}g_{ii}$) is diagonal, and $A$ also.

Moreover, we can introduce a local function (metric potential) $\eta$, defined up to addition of a constant, such that: $g_{ii}=e_i\eta$.

\medskip

\subsection{The space of polynomials}
The space of complex polynomials is an example of semisimple Frobenius manifolds of arbitrary dimension (the construction is due to Dubrovin~\cite{Du} and Saito~\cite{Sa82}).
Consider the $n$-dimensional affine space $\mathbb{A}^{n}$ with coordinate functions $a_1, ..., a_n$. 
We identify the space $\mathbb{A}^{n}$ with the space of degree $n+1$ polynomials $P(z)=z^{n+1}+a_1z^{n-1}+\dots + a_{n}$. 
Let \[\pi: \tilde{\mathbb{A}}^{n} \to \mathbb{A}^{n}\] be a covering space of degree $n!$ whose fiber over a point $P(z)$ consists of the total orderings of the roots of $P'(z)$. In other words, $\tilde{\mathbb{A}}^{n}$ supports functions $\rho_1,...,\rho_n$ such that: 

\[\pi^*(P'(z))= (n+1)\prod_{i=1}^{n}(z-\rho_i);\]

\[\pi^{*}(a_i)=(-1)^{i+1}\frac{n+1}{n-i}\sigma_{i+1}(\rho_1,...,\rho_n), i=1, ...,n-1\]
and $\sigma_1(\rho_1,...,\rho_n)= \rho_1+ ...+\rho_n=0$. We will omit $\pi^*$ in the notation of lifted functions.

Let $ \Mm \subset \tilde{\mathbb{A}}^{n}$ be the open dense subspace in which : 
\begin{enumerate}
\item $\forall i$, $P''(z_i) \neq 0$ that is $\rho_i \neq \rho_j$ for $i\neq j$. 
\item $u^{i}:= P(\rho_i)$ form local coordinates at any point. 
\end{enumerate}

\vskip.1cm 

$\Mm$ is a {\it semisimple Frobenius manifold} with the following structure data: 

\begin{enumerate}
\item The $n$-tuple $(u^i)$, generate the canonical coordinates, the basis is given by $e_i=\frac{\partial}{\partial u^i}$, and the dual basis is generated by $(du^i)$. 
\item Flat metric \[g=\sum \frac{(du^i)^2}{P''(\rho_i)}\] 
with metric potential \[\eta=\frac{a_1}{n+1}= \frac{1}{n-1}\sum_{i<j} \rho_i\rho_j= \frac{1}{2(n-1)}\sum \rho_i^2.\] 
\end{enumerate}
 
\medskip
\begin{thm}
The Saito semisimple Frobenius manifold, has a cell decomposition, induced from the Gauss-skizze.
\end{thm}
\begin{proof}
Let $\tilde{\mathbb{A}}^{n}$ be the space of critical points of the space of polynomials of degree $n$. Let $m_1,\dots, m_k$ be a set of natural numbers, summing to $n-1$. We define $e(m_1,\dots, m_k)$ the subset of $ \tilde{\mathbb{A}}^{n}$ consisting of points $\rho_1,...,\rho_{n-1}$ such that they lie on $k$ pairwise distinct lines. On every line, the critical points collide (and thus are at the intersection of red or blue curves). On every $i$-th line, there exist $m_i$ critical points (not all distinct). Let us suppose that the $i$-th line contains $j_i$ different sets of collided points. Let $j_1+...+j_k=q$

\smallskip 

We define the space $W_r$ as the space of rows $(w_1,...,w_r)$ where $w_i \in \Rr\cup\{\infty\}$ and $w_1\leq w_2\leq \dots \leq w_r$. It is clear that the space $W_r$ is isomorphic to an $r$-dimensional cube. The space $e(m_1,\dots, m_k)$ is a cell of dimension $q+k$. The characteristic mapping $f:I^{q+k}\to \tilde{\mathbb{A}}^{n}$ is given by the formula $f:I^{q+k}=W_k\times W_{j_1}\times\dots W_{j_k} \to \tilde{\mathbb{A}}^{n}$.
\smallskip

The characteristic map $f$ maps homeomorphically the interior of the cube $:I^{q+k}$ on $e(m_1,\dots, m_k)$.\end{proof}

\medskip
\subsection{Topological stratification}

\medskip

\begin{defn}\label{D:strat}
Let $X$ be a topological space equipped with a stratification $\mathcal{\tilde{S}}$, and let $n = max\{ dim\, S\, |\, S \in \mathcal{\tilde{S}}\}$. The stratification $\mathcal{\tilde{S}}$ is called a topological stratification if it satisfies the following condition: for any point $x$ in a stratum of dimension $i$ with $i < n$, there exists a distinguished neighborhood $N$ of $x$ in $X$, a compact Hausdorff space $L$ with an $n-i-1$ dimensional topological stratification 
\[L=L^{n-i-1}\supset \dots \supset L^1\supset L^0\supset L^{-1}=\emptyset\]
and a homeomorphism 
\[\phi:\Rr^i\times cone^\circ(L)\to N,\] which takes $\Rr^i \times cone^{\circ} (L^j) \to N\cap X_{i+j+1}$. The symbol $cone^{\circ} (L^j)$ denotes the open cone, $L\times [0,1)/(l,0)\sim(l',0)$ for all $l, l'\in L$.
\end{defn}

\medskip

Let $\mathcal{S}$ be the stratification of $\overline{\textup{Conf}}_n(\C)$. A stratum in $\overline{\textup{Conf}}_n(\C)$ is a connected component formed from all degree $n$ polynomials, such that their Gauss-skizzes are all isomorphic to a geometric realisation of a Gauss-graph $\sigma$. A stratum (or connected component) of $\overline{\textup{Conf}}_n(\C)$ is denoted $A_\sigma\subset \overline{\textup{Conf}}_n(\C)$, where $\sigma\in\GG_n$. In particular, we have $$\mathcal{A}_{\overline{\sigma}}=\cup_{\sigma \prec \tau} \mathcal{A}_{\tau},$$
which leads to the following statement. 

\medskip
\begin{lem} 
The functor from the category $\tilde{P}$ to the category $(\mathcal{S}, \subset)$ is fully faithful.
\end{lem}
\begin{proof}
The union $\overline{\sigma}=\{\tau : \sigma \prec \tau\}$ corresponds bijectively to $\mathcal{A}_{\overline{\sigma}}=\cup_{\sigma \prec \tau} \mathcal{A}_{\tau}.$ 
In particular, for any $\sigma\prec \tau$ there exists a unique $\mathcal{A}_{\sigma}$ and a unique $\mathcal{A}_{\tau}$ such that $\mathcal{A}_{\tau}\subset \overline{A}_{\sigma}$.
\end{proof}
\begin{cor}
There is an isomorphism of categories $\tilde{P}$ and $(\mathcal{S}, \subset)$.
\end{cor}

\medskip

\begin{prop}\label{P:up}
Consider the poset $\tilde{P}$. Then, there exists one unique lower bound. 
\end{prop}
\begin{proof}
The proof is by induction. 
One can easily verify that the poset of $2$-Gauss-graphs endowed with $\prec$ order relation, has a unique upper bound: it is a tree $T_1$, with $8$ leaves (of alternating color) and which contains one unique inner node of valency $8$.
Consider the poset of 3-Gauss-graphs and consider those graphs containing $T_1$ as a subgraph. 

\smallskip

Applying a contraction $WH_-$ move onto the inner node of $T_1$ and onto the inner node of the remaining tree, we have a new tree with 12 leaves and a unique inner node of valency 12. It is impossible to apply again a $WH_-$ onto the new tree. Therefore, it is a unique upper bound. 

\smallskip 

Suppose that for the poset of $n$-Gauss-graphs endowed with $\prec$ order relation, there exists a unique upper bound, being a tree $T_1$ with $4n$ leaves (of alternating color), and containing one unique inner node (of valency $4n$). 

\smallskip 

Let us show that for the poset of $(n+1)$-Gauss-graphs, there exists a unique upper bound too. Consider the  $(n+1)$-Gauss-graphs containing the tree $T_1$ and an extra tree $T_2$. The tree $T_2$ has four leaves of alternating colors, and one inner node. 
\smallskip

Apply the $WH_-$ contraction operation (or a composition of two contraction morphisms), so that the inner nodes of $T_1$ and $T_2$ give a new tree, with $4(n+1)$ leaves and one (unique) inner node. Applying again a $WH_-$ move is impossible. Therefore, this new tree is the ()unique) upper bound. 
\end{proof}

\medskip

\begin{cor}
Consider the stratification by Gauss-graphs. Then, there exist a unique upper bound. 
\end{cor}

\medskip

\begin{lem}
Let us consider the functor mapping the poset $\tilde{P}$ to the geometric realisation of the simplicial complex, whose vertices are the elements of $\tilde{P}$ and whose simplices correspond to finite chains (totally ordered subsets) of $\tilde{P}$. Then, this geometric realisation is contractible.
\end{lem}
\begin{proof}
Since, by proposition~\ref{P:up} there exists one unique upper bound, it is a classical result that the geometric realisation of the simplicial complex corresponding to $\tilde{P}$ is contractible (see~\cite{Wa04}, p.8).
\end{proof}

 %%%%%%%%%%%%%%%%
\section{Deformations, Maurer--Cartan equations. New perspectives}
In this section we deform Gauss-skizze. We present first the classical approach, which consists in using deformation theory, and we show that this can be used only locally in the context of the Gauss-skizze. 

\smallskip 
 
Informally speaking, the aim of deformation theory is to study structures of a given type, existing on an object up to some 'equivalence'. Deformations are ruled by the so-called {\sl Maurer--Cartan equations} and by differentially graded Lie algebras (DGLA).

\smallskip 

Since we investigate the deformation problems on the space of Gauss-skizze, therefore we work on the field $\Rr$ and will give the explicit DGLA. 

\smallskip

For a global study of the deformations of these curves, it is necessary to use a complementary viewpoint, which is ruled by {\sl Hamilton--Jacobi equations}. This will be shown in the second part of this section.

\medskip
\subsection{Deformations, Maurer--Cartan equations}
Classical deformation theory for algebraic-geometry is based on the work of Kodaira--Spencer~\cite{KoS58} and Kuranishi~\cite{Ku68} on small deformations of complex manifolds. It was formalised by Grothendieck in the language of algebraic schemes~\cite{Gro95}.

 \smallskip

The idea by Grothendieck~\cite{Gro95}, M. Schlessinger~\cite{Sch68} and M. Artin ~\cite{Ar76} is that, an infinitesimal deformation of an algebraic object $X$ over a field $k$ is a local deformation of $X$ parametrized by an algebra $R$ in the category $ {\bf Art}$ of Artinian rings. 

\smallskip

Le $Def_X(R)$ be the set of equivalence classes of local deformations of $X$, parametrized by $R$. Then, the assignment $R\mapsto Def_X(R)$ defines the functor: 
\[Def_X: {\bf Art}\to {\bf Set}\] called the classical deformation functor. 

The classical formalism of Grothendieck--Mumford--Schlessinger of infinitesimal deformation theory lead to the following mantra:

\smallskip 

\[\text{Deformation\ problem} \longleftrightarrow \text{Deformation\ functor/groupoid}\]

\smallskip 

P. Deligne formulated the following philosophy:
{\it ''In characteristic 0, any deformation problem is controlled by DGLA (Differentially Graded Lie Algebra)."}
This was proven in the last years using homotopic algebra, and higher categories (categories of models and $\infty$-categories see~\cite{Lu1} and~\cite{Pr1}).
 
 \smallskip 
 
 This result had as a result that once a DGLA ${\bf L}$ is given, one can define the associated deformation functor $Def_{\bf L} : {\bf Art} \to {\bf Set}$, using the solutions of the Maurer--Cartan equation up to gauge equivalence. More precisely:

\medskip 

\begin{defn}[\cite{Ia10}, Definition 1.6]
Let ${\bf L}$ be a DGLA. Then, the Maurer--Cartan functor associated with ${\bf L}$ is the functor 

\[MC_{\bf L} :{\bf Art} \to {\bf Set},\] 
\[MC_{\bf L}(A)= \{x \in {\bf L}^1\otimes \m_A |dx+\frac{1}{2}[x,x]=0\},\]
where $\m_A$ is the maximal ideal for $A\in {\bf Art}$. 
\end{defn}

\medskip
The deformation functor $Def_{\bf L}$ associated with a DGLA ${\bf L}$ is given by the quotient of $MC_L(A)$ by the the group:
\[Def_{\bf L}=\frac{\{x\in {\bf L}^1\otimes \m_A| dx+\frac{1}{2}[x,x]=0\}}{exp({\bf L}^0\otimes \m_A)}\]
\smallskip

Since we are interested in the field of real numbers, we use the construction of the DGLA, given in~\cite{Ia10} for the deformation. The big picture is as follows. 
For $X$ a smooth variety, defined over field of real numbers, define the tangent sheaf $\Theta_X$. Given, $Z \subset X$ be a closed subscheme, define $\Theta_X(-logZ)$ to be the sheaf of tangent vectors to $X$, being tangent to $Z$. The inclusion of sheaves of Lie algebras is $\Xi: \Theta_X(- log Z ) \hookrightarrow \Theta_X$.

\smallskip

Let $\mathcal{U}$ be an affine open cover of $X$. We can associate a \v Cech cosimplicial Lie algebra by putting $\Theta_X(\mathcal{U})$. This is defined by introducing a sequence of Lie algebras $\{\g_k = \prod_{i_0<\dots<i_k} \Theta_X(U_{i_0\dots i_k})\}$ and maps among them. In particular, $\g_0 = \prod_i \Theta_X(U_i)$ and each $\Theta_X(U_i)$ control infinitesimal deformations of the open set $U_i$, whereas maps control the gluing of deformations. 

\smallskip

The main point of the construction is that there exists a semicosimplicial differential graded Lie algebra, $\g = \{\g_k\}_k$, with $\g_0$ controlling the deformations of each open set of the cover. 

\smallskip

Associating to $\Theta_X(-logZ)$ and $\Theta_X$, the \v Cech semicosimplicial
Lie algebras $\Theta_X(-logZ)(U)$ and $\Theta_X(U)$, respectively, this leads to introducing a bisemicosimplicial \footnote{Let $\Delta_{\bullet}$ be the category defined in~\ref{S:1}. An object in a category {\bf C} is a bisemicosimplicial object is a covariant functor $\Delta_{\bullet}\times \Delta_{\bullet}\to {\bf C}$} Lie algebra, given by: $\Xi^{*}: \Theta_X(-logZ)(U) \to \Theta_X(U)$.

\smallskip

It is now sufficient to find a differential graded Lie algebra, by using the Thom--Whitney construction, giving the DGLA $Tot^*_{TW}(\Xi^*)$ (see \cite{Ia10} for more). This algebra controls the deformations of the closed subscheme $Z$. 

\medskip 

 \subsection{Deformation of Gauss-skizze}
 
In the following paragraphs, we remedy to the Gauss-skizze deformation problem by adopting a complementary point of view on deformation problems.

\smallskip

The classical method of investigating deformation problems in algebraic-geometry fails to work for the study of all deformations of Gauss-skizze. The only case when this operating mode is possible concerns polynomials of type $z^n-c$, where $c\in \C$. We expose an example for degree 2 polynomials of type $z^2-c$.

\smallskip 

\begin{example}
Take the complex polynomial $P(z)=z^2+a_0$, where $a_0$ is a real number. In real variables, the polynomial $P$ is splitted into $(x^2-y^2+a_0)+\imath2xy$. We focus on the harmonic polynomial given by the imaginary part and its patterns. 

\smallskip

Consider the $\Rr$-algebra $A=\Rr[x,y]/(xy)$. It can be deformed to $A_t=\Rr[x,y]/(xy-t)$, for any value $t\neq 0$ of the parameter $t$ in $\Rr$. The family $\{A_t: t\in \Rr\}$ of {\it commutative $\Rr$-algebras} can be considered as a family of deformations of the commutative $\Rr$-algebra $A=A_0$. 
\smallskip

Concerning $A_0$, the pattern in $\Bp_n$ is a pair of two lines, coinciding with the real and imaginary axis and intersecting at the origin. The deformed patterns at $A_t$ give a hyperbola, i.e. a pair of smooth curves. Combinatorially we have proceeded to an expanding $WH_+$ move. 

\smallskip
 
This example illustrates the notion of {\it local deformations} of the patterns and of the algebra behind. Let $R=\Rr[[t]]$ be the formal power series ring on one variable. Then, $A_R=\Rr[x,y][[t]]/(xy-t)$ is a {\it local deformation of $A$ parametrized by $R$}. Indeed, $A_R$ is $R$-flat with isomorphism $\Rr\otimes A_r \to A$ induced by the map $A_R \to A$ given by $t\to 0.$ 
Similar investigations can be lead for s $(x^2-y^2+a_0)$. 
\end{example}

\smallskip 

Somehow, the problem of deforming the Gauss-skizze of more general polynomials fails to be algebraically explicit with this method. Indeed, let us consider the coordinate ring $A=\Rr[x,y]/(ReP(x,y))$. This polynomial has a given number of critical points of $P$ lying on $ReP(x,y)=0$ and corresponds to a given pattern. The aim is to locally deform the original structure, by adding new critical points or smoothing those already pre-existing. It is clear that an explicit parametrization $R$ here is not possible. 

\medskip 

\begin{thm}\label{P:germ}
Let $\mathcal{C}$ be a pattern in $\Bp_n$ (or $\Rp_n$) and consider one connected component in $\mathcal{C}$. Then, there exists an explicit germ corresponding to it and an explicit local deformation parametrized by an object of ${\bf Art}$. \end{thm}
Before we prove the theorem, let us recall a statement by Kas $\&$ Schlessinger from~\cite{KS72}.

\smallskip 

For $V_0\subset \C^n$ an analytic variety of dimension $n-p$. Kas and Schlessinger give an analytic construction for a versal deformation of its germ at 0, for the case where $V_0$ has an isolated singularity at 0, and $V_0$ is a local complete intersection at 0. 
\begin{thm*}[Kas $\&$ Schlessinger, \cite{KS72}]\label{T:KS}
The family $\dot{\pi}:(V,0)\to (\C^l,0)$ is a versal deformation of the germ $(V_0,0)$; that is, any flat deformation $\Psi : (W,0)\to (S,0)$ of $(V_0,0)$ is induced from $\dot{\pi} (V,0) \to (\C^l,0)$ by a map $\varphi: (S,0) \to (\C^l,0)$.
\end{thm*}

\begin{proof}

Consider a connected component in $\mathcal{C}$. It is an embedded tree in $\Rr^2$. 
Components of the complement of the embedded tree are called regions. All faces are simply connected, unbounded regions. 

\smallskip 
In the following, we mimic the method developed in~\cite{ErJaNa07} (in the proof of theorem 3).
An {\it indexing} is added to the tree, i.e. a prescription of positive numbers to the edges, such that: for each face, the sum of these numbers over all edges on the boundary of this face equals to $2\pi$ times the degree of the face. It can be shown that for every forest there exists an indexing. In the special case that there are no vertices the indexing is unique and the index of each edge is $2\pi$.

\smallskip

Let us apply the following well known construction. Let $\mathfrak{e} = [v_0,\infty)$ be an edge, whose one extremity is $\infty$ and let us choose a point $\tilde{v}$ in the interior of this edge.
We replace in $\mathfrak{e}$ the part $[\tilde{v},\infty)$ by an odd number of edges, that we denote $e_1,...,2_{2k+1}$ and whose extremities are $\tilde{v}$ and $\infty$. Now, by this construction the edge $\mathfrak{e} $ is replaced by $e_0=[v_0,\tilde{v}]$ and $e_1,...,2_{2k+1}$ (where it is supposed that $e_0,...,2_{2k+1}$ are labeled in a natural cyclic order around $v$).

\smallskip 
If  the label of $\mathfrak{e}$ is $\kappa$, then in the new tree, we index $e_0$ by $\frac{1}{2}\kappa$ 
and the remaining edges by $\frac{1}{2}\kappa$ and $2\pi-\frac{1}{2}\kappa$ alternatively. Using this, we define the length of a path in the forest as the sum of the indexes of edges in this path. 

\smallskip 

To this new construction, an orientation is added to the set of edges. Supposing that the set of edges $e_0,...,2_{2k+1}$ attached to the vertex $\tilde{v}$ are labeled in a natural cyclic order, for any pair $e_j$ and $e_{j+1}$ one edge is oriented towards $\tilde{v}$, the other one is oriented in the other direction. There are exactly two ways of putting an orientation. This induces orientation on the boundaries of each face having common boundary with this tree. So, we obtain orientation of all other trees which intersect the boundary of these faces. Continuing this procedure we orient the edges of the whole forest, and in particular the boundaries of all faces.

\smallskip 

Now we construct a ramified covering from the forest to the real line $\Rr$ and we equip $\Rr$ with the spherical metric $2|d\kappa|/(1+|\kappa|^2)$. The length of $\Rr$ is thus $2\pi$. The real line has a natural orientation from $-\infty$ to $\infty$.
Each edge of the forest is then mapped homeomorphically and respecting the orientation (that is increasing with respect to the orientations of the edge and of the real line) onto an interval of $\Rr$. All leaves are mapped to $\infty$. Such a continuous map is uniquely defined for fixed indexes and orientations on the edges. It is a ramified covering (ramified at the vertices).

\smallskip 

Since this gives an explicit algebraic curve germ $A_0$ with isolated singularity, we can apply easily the Theorem~\ref{T:KS} developed by in Kas and Schlessingerand in~\cite{KS72}, and the statement in Fox~\cite{Fo93}.This   locally deforms the germ $A_0$ into $A_R$. This deformation is parametrized by an algebra $R$.

\end{proof}

\medskip 

\begin{cor}
Let $\mathcal{C}$ be a pattern. Consider one connected component in $\mathcal{C}$ being a tree with non-empty set of inner nodes. An expanding $WH_+$ move applied to an inner node amounts to making a versal deformation of the germ associated to this inner node. 
\end{cor}
This result is straightforward from~\cite{KS72} and the exposition in~ \cite{Fo93}.

\medskip 
\subsection{New approach}
To overcome the problem, highlighted above, we propose a complementary approach, to the classical Maurer--Cartan perspective. We have the following diagram:
\[\begin{tikzcd}
 \GS_n \arrow[d,"\pi"]& \GG_n \arrow[ld,dashed] \arrow[l,"\rho"] \\
 \mathcal{T}
 \end{tikzcd}\]
where $ \mathcal{T}$, isomorphic to the affine space $\Rr^{2n}$, which is parametrizing $ \GS_n$, and $\pi$ is a proper flat morphism, $\rho$ is an isomorphism.

\smallskip 

A path in the space of degree $n$ complex polynomials, in one variable is denoted by $P(z,t)$, where $t\in [0,1]$.
Similarly, a path on the space of patterns will be denoted by $ReP({\bf x},t)=0$ (resp. $ImP({\bf x},t)=0$). This is equivalent to writing $P^{-1}(\Rr,t)$ respectively $P^{-1}(\imath\Rr,t)$.

\smallskip

Any point in the fiber of $\pi$ is a Gauss-skizze i.e. a pair of level curves given by $ReP({\bf x},t)=0$ and $ImP({\bf x},t)=0)$, where $P$ is a given complex polynomial, its coefficients lying in $T$. Consider a pair of distinct fibers. Since the space of polynomial is connected, there always exists a path from one fibre to the other one.

\medskip

\begin{defn}
A deformation of a Gauss-skizze is trivial whenever the deformed Gauss-skizze is isomorphic (in the graph sense) to the initial one.
We say that a deformation of a Gauss-skizze is a non-trivial small deformation whenever it is not trivial. 
\end{defn}

\medskip
\begin{thm}
A deformation of any Gauss skizze obeys to the following differential equation:
\begin{equation}\label{E:HJ}
\frac{ \partial}{\partial t}{\Phi({\bf x},t)}+ {\bf v}\nabla \Phi({\bf x},t)=0,
\end{equation}
where $\Phi$ is given by $\Phi=P^{-1}(\Rr,t)$ or $\Phi=P^{-1}(\imath\Rr,t)$, for a given complex polynomial $P$, $t\in [0,1]$ and {\bf v} is a vector field.
\end{thm}
\begin{proof}
This follows from the level curves method (chapter 1 in~\cite{OsPa03}) and by applying the Cauchy--Riemann equations to $ReP$ and $ImP$. 
\end{proof}

\smallskip 

\section{Groupoids of Gauss-skizze and their deformations}
Let $\tilde{P}$ (resp. $\tilde{\dot{P}}$) be the poset $( \GG_n,\preceq)$ of Gauss-graphs (or forests in (${\bf TrCoGr},\sqcup)$, having an even number of leaves and nodes of even valency) and endowed with order relation $\preceq$. This poset is viewed as a category in the standard way, i.e., with elements being the objects (Gauss-graphs or forests) and order-relations $'\prec'$ being morphisms. Let furthermore ${\bf Top}$ denote the category of topological spaces and continuous maps.

\smallskip 

Let us introduce the functor from the category $\tilde{P}$ to the category of topological spaces $\mathcal{A}: \tilde{P} \to {\bf Top}$, giving the following construction. 

\begin{enumerate}
\item A fiber, at a given point $P\in \mathcal{A}_{\sigma}\subset \mathcal{T}$, is a Gauss-skizze. 
\item For a given graph (Gauss graph or forest such as described above), there exists a stratum $\mathcal{A}_{\sigma}$ of $\mathcal{T}$ in ${\bf Top}$, such that 
any element $P\in \mathcal{A}_{\sigma}$ is associated to a Gauss-skizze (or pattern) of combinatorial type $\sigma$.
\item $\overline{\textup{Conf}}_n(\C)$ is partitioned by the (disjoint) union of all $\mathcal{A}_{\sigma}$. The closure of any stratum $\mathcal{A}_{\sigma}$ is the union of all strata $\mathcal{A}_{\tau}$ such that $\sigma\prec \tau$ in the sense of the definition \ref{D:WH} and the discussion following it. 
\item If $\sigma\in Ob( \GG_n)$ and $A(\sigma)=\mathcal{A}_{\sigma}$, we will call $\sigma$ a family with base $\mathcal{A}_{\sigma}$, or a $\mathcal{A}_{\sigma}$-family.

\end{enumerate}

\medskip

\begin{prop}
The triple formed by the category $( \GG_n,\prec)$, the category $(\mathcal{T},\subset)$ and the functor $\mathcal{A}: ( \GG_n,\prec) \to (\mathcal{T},\subset)$ forms a groupoid. 
\end{prop}
\begin{proof}
According to~\cite{Ma99} (chap.5, section 3.2) in order to form a groupoid, the data must satisfy one of the following conditions: 
\begin{enumerate}
 \item First, for any base $\mathcal{A}_{\sigma}\in {\bf Top}$, any morphism of families over $\mathcal{A}_{\sigma}$ inducing identity on $\mathcal{A}_{\sigma}$ must be an isomorphism.
 \item For any arrow $\psi:\mathcal{A}_{\sigma_1}\to \mathcal{A}_{\sigma_2}$, between the basis and any family $\sigma_2$ over the target $\mathcal{A}_{\sigma_2}$, there must exist a $\mathcal{A}_{\sigma_1}$-family $\sigma_1$ and a morphism $\sigma_1\to \sigma_2$ lifting $\psi$. 
\end{enumerate}
\smallskip 
Let us apply condition 2 to our context. We have morphisms $\phi:\sigma_1\to \sigma_2$ and $\psi: \mathcal{A}_{\sigma_1}\to \mathcal{A}_{\sigma_2}$ such that $\psi$ induces an isomorphism of graphs $\sigma_1\to \psi^*(\sigma_2).$
Equivalently the diagram

\begin{equation}\label{E:Cartesian}
\begin{tikzcd}
\sigma_1 \arrow{r}{\phi} \arrow{d}{} & \sigma_2 \arrow{d}{} \\
\mathcal{A}_{\sigma_1} \arrow{r}{\psi} &\mathcal{A}_{\sigma_2}
\end{tikzcd}
\end{equation}

is Cartesian and induces the bijection of two families of labeled sections. 
\end{proof}

\smallskip 

Consider $\sigma$ an element of $ \GG_n$ (an object of $\tilde{P}$) and the fiber of $ \GG_n$ over the stratum $\mathcal{A}_{\sigma}$, which is $\mathcal{A}(\sigma)$. 
If $\sigma,\tau \in P$, $\sigma \leq \tau $, we use $\mathcal{A}(\sigma \to \tau)$ to denote the continuous map associated to the order relation $\sigma \prec \tau$ (a morphism in $\tilde{P}$). 

\smallskip 

In particular, we have the following commutative diagram: 

\[
\begin{tikzcd}
\sigma \arrow{r}{\preceq} \arrow{d}{} &\tau \arrow{d}{} \\
\mathcal{A}_{\sigma} \arrow{r}{\psi} & \mathcal{A}_{\tau} 
\end{tikzcd}
\]

\medskip
\begin{thm}
Let $\sigma, \tau\in Ob(\F)$ such that $\sigma \prec \tau$. The continuous map $\mathcal{A}(\sigma \to \tau)$ is a deformation retract of $\mathcal{A}_{\sigma}$ onto $\mathcal{A}_{\tau}$.\end{thm}

\begin{proof}
Let $I=[0,1]$ be an interval in $\Rr$. For a morphism $\psi : \mathcal{A}_{\sigma}\to \mathcal{A}_{\tau}$, the mapping cylinder is the quotient space of the disjoint union $(\mathcal{A}_{\sigma}\times I)\sqcup \mathcal{A}_{\tau}$ obtained by identifying each $(x,1) \in \mathcal{A}_{\sigma}\times I$
with $\psi(x)\in \mathcal{A}_{\tau}$. 

\smallskip 
A mapping cylinder $M_\psi$ produces a deformation retract to the subspace $\mathcal{A}_{\tau}$, by sliding each point $(x, t)$ along the segment $\{x\}\times I \subset M_\phi$ to the endpoint $\psi(x) \in \mathcal{A}_{\tau}$ .

\smallskip

We explicitly construct the deformation retract operation. Take a fiber of $\mathcal{A}({\sigma})$, i.e. a Gauss-skizze parametrized by the topological space $\mathcal{A}_{\sigma}$. 
The continuous map defining $\mathcal{A}(\sigma\to \tau)$ is given by the path in the space of degree $n$ polynomials $ReP({\bf x},t)=0$ (resp. $ImP({\bf x},t)=0$), where $t\in[0,1]$. For simplicity, we consider only one level curve, for instance $P^{-1}(\Rr)$ (resp. $P^{-1}(\imath\Rr)$).

 \smallskip
 
 The deformation retract is constructed by convecting the $ReP(\bf{x},t)=0$ (resp. $ImP(\bf{x},t)=0$) curves with the vector field ${\bf v}$. This is given by the equation \ref{E:HJ} of the type 
\[\frac{ \partial}{\partial t}{\Phi({\bf x},t)}+ {\bf v}\nabla \Phi({\bf x},t)=0,\]

 where $\Phi({\bf x},t)=P^{-1}(\Rr,t)$ (resp.$\Phi({\bf x},t)=P^{-1}(\imath\Rr,t)$).
 
\smallskip 

We construct the deformation such that at $t=1$, there exists an extra set of points $\{\rho'_i\}_{i\in J}$, $J \in Ob({\bf Fin})$ verifying $\nabla ReP(\rho'_i,1)=0$, resp. $\nabla ImP(\rho'_i,1)=0$. In other words, we have that $ReP(\rho'_i,1)=0$ and $P'(\rho'_i,1)=0$, respectively $ImP(\rho'_i,1)=0$ and $P'(\rho'_i,1)=0$. 

\smallskip 

This is a deformation retract of $\sigma$ onto $\tau$, since we have constructed a local cone defined by $[0,1]\times \{P^{-1}(\Rr,t)\}$, with apex given at $t=1$, located at the point where the curves have shrinked together. Using the cartesian diagram \ref{E:Cartesian} it follows that we have a deformation retract of $\mathcal{A}_{\sigma}$ onto $\mathcal{A}_{\tau}$.
\end{proof}

\medskip

\begin{thm}
The stratification of $\overline{\textup{Conf}}_n(\C)$ by Gauss-skizze forms a topological stratification.
\end{thm}
\begin{proof}
Consider a point $x$ of a stratum of dimension $k>0$. By theorem~\ref{P:germ}, for a given tree $T_1$ in $\sigma$ (which contains more than one edge) we have shown that there exists a versal deformation of the germ corresponding to that tree, which smoothens the critical point(s). Locally, for $U$ in a neighborhood of $x$ and in $\textup{Conf}_n(\C)$ this procedure is equivalent to having $U\cong C \times \Rr$ , where $C$ is a topological space of dimension $n-k$. 

\smallskip 
Let us show that $C$ is a cone over a "link" $L$ of dimension $n-k-1$, and such that $Y$ is a topologically stratified space with strata coinciding with those of $X$ such as depicted in definition~\ref{D:strat}. 
Consider all the trees in $\sigma$, which are not $T_1$, i.e. $\sigma\setminus T_1$. We can consider it as a new graph. By proposition~\ref{P:up} for $m$-Gauss-graphs, there exists a unique maximal upper bound in the poset $( \GG_n, \preceq)$. This maximal upper bound is a Gauss-graph constituted from one tree with $4m$ leaves of alternating colors, and with one unique inner node. Using theorem~\ref{P:germ} this corresponds to a complex germ, which can be explicitly given as $(z^{n-k},0)$. 

\smallskip 

Using Milnor's cone theorem (Chap. 4~\cite{Mi68}), the intersection of a ball with this germ (such that the center of the ball coincides with the critical point), is topologically equivalent to a cone over the link $L$ of the germ (the link being the intersection of the boundary of the ball with the germ). Following the stratification rule given by the Gauss-graphs, $L$
inherits the Gauss-graphs' stratification. 
\end{proof}

\subsection{Gauss-skizze operad}\label{S:Gd}

We propose here to enrich the Fulton--MacPherson operad, by considering it through the angle of the Gauss-skizze. In particular, every object $\overline{\textup{Conf}}_n(\C)$ has a topological stratification by Gauss-graphs. We will use the tools from section~\ref{S:op} to construct the composition operation. The advantage of considering this, is that it gives a refined (semi-algebraic) topological decomposition for $\overline{\textup{Conf}}_n(\C)$. Each object of $\mathcal{P}^\bullet$ is stratified by the Gauss-skizze. 
 
\begin{defn}
${\bf GaSop}$ is the collection of objects $\{\mathcal{V}(S)| S\in {\bf Fin}\}$, where $\mathcal{V}(S)$ is the stratified space $\textup{Conf}_{n}(\C)$ in $n$-Gauss-graphs and where $Card(S)=n$.

\end{defn}

\begin{thm}
${\bf GaSop}$ forms a topological (symmetric) operad.
\end{thm}
\begin{proof}
Most of the properties of ${\bf GaSop}$ are inherited from the pre-exiting Fulton--MacPherson operad. What remains to discuss concerns the 
composition operation, with the Gauss-skizze tool. We rely on the cosimplicial construction of section~\ref{S:op} and in particular definition~\ref{D:co}.

\smallskip 

Let $Q$ be the linear set, defined in section~\ref{S:op}. Consider a point in $\textup{Conf}_Q$ i.e. corresponding to a set of marked points plane on $\Rr^2$. Here, with ${\bf GaSop}$ there exists an additional piece of information: the topological realisation of a Gauss graph $\sigma^n_0$. The inner vertices of this graph, lying in $V_{roots}$, coincide with the marked points of configurations space.
 \smallskip 
 
Let us consider a weak partition $\tilde{v}:S \to Q$, and determine $\tilde{v}^{-1}(p)$, for all $p\in Q$. The cardinality of $\tilde{v}^{-1}(p)$ gives the multiplicity of the $p$-th point in the Fulton--MacPherson compactified space $\overline{\textup{Conf}}_Q$. 

\smallskip 

Going back to the topological realisation of a graph $\sigma^n_0$, we will modify it according to the construction presented in~\ref{S:op}. Add to this $p$-th point new edges, so that $4| \tilde{v}^{-1}(p)|$ is the new valency of $p$. Colors/orientations of edges incident to $p$ alternate. This new embedded graph is a Gauss-skizze. Note that this operation is well-defined and defines uniquely a polynomial with multiple roots and therefore corresponds to a unique configuration of points. Applying an expanding $WH_+$ move to $p$, splitting $p$ into $\tilde{v}^{-1}(p)$ new inner nodes. This splitting can be such that all trees are disjoint or, trees can be connected. 
\end{proof}

 \begin{cor}
The topological operad ${\bf GaSop}$ is weakly equivalent in the model structure on operads with respect to the classical model structure on topological spaces, to the little $2$-disk operad.
\end{cor}

\end{document}